\documentclass[11pt]{amsart}
\usepackage{amsthm}
\usepackage{amsmath,amstext,amsthm,amsfonts,amssymb,amsthm}
\usepackage{mathrsfs}
\usepackage{multicol}
\usepackage{latexsym}
\usepackage{color}
\usepackage{graphicx}
\usepackage{epsf}
\usepackage{epsfig}
\usepackage{epic}
\usepackage{graphics}
\usepackage{verbatim}
\usepackage{color}
\usepackage{hyperref}
\usepackage{bm}
\newtheorem{theorem}{Theorem}[section]

\theoremstyle{definition}
\newtheorem{definition}[theorem]{Definition}
\newtheorem{proposition}[theorem]{Proposition}

\theoremstyle{remark}
\newtheorem{remark}[theorem]{Remark}
\numberwithin{equation}{section}

\newcommand{\q}{\mathfrak{q}}
\newcommand{\HI}{\mathfrak{H}}

\newcommand{\B}{\mathcal{B}}

\newcommand{\qu}{\mathfrak{q}}
\newcommand{\pu}{\mathfrak{p}}
\newcommand{\oqu}{\overline{\mathfrak{q}}}

\newcommand{\M}{\mathfrak M}
\newcommand{\quat}{\mathbb H}
\newcommand{\R}{\Bbb R}

\newcommand{\mc}{\mathcal}

\newcommand{\be}{\begin{equation}}
\newcommand{\en}{\end{equation}}
\newcommand{\D}{{\mc D}}

\newcommand{\bedefin}{\begin{defi}}
	\newcommand{\findefi}{\end{defi} \medskip}
\newcommand{\betheo}{\begin{theorem}$\!\!${\bf \,\,\,}}
	\newcommand{\entheo}{\end{theorem}}
\newcommand{\enth}{\end{theorem}}
\newcommand{\becor}{\begin{cor}$\!\!${\bf .}}
	\newcommand{\encor}{\end{cor}}
\newcommand{\belem}{\begin{lem}$\!\!${\bf .}}
	\newcommand{\enlem}{\end{lem}}
\newcommand{\bea}{\begin{eqnarray}}
\newcommand{\ena}{\end{eqnarray}}
\newcommand{\beano}{\begin{eqnarray*}}
	\newcommand{\enano}{\end{eqnarray*}}
\newcommand{\bee}{\begin{enumerate}}
	\newcommand{\ene}{\end{enumerate}}
\newcommand{\bei}{\begin{itemize}}
	\newcommand{\eni}{\end{itemize}}
\newcommand{\betab}{\begin{tabular}}
	\newcommand{\entab}{\end{tabular}}

\newcommand{\Iop}{{\mathbb{I}_{V_{\mathbb{H}}^{R}}}}

\newcommand{\bfraka}{\mbox{\boldmath $\mathfrak a$}}
\newcommand{\bfrakb}{\mbox{\boldmath $\mathfrak b$}}

\newcommand{\bfrakp}{\mbox{ $\mathfrak p$}}

\newcommand{\bk}{\mathbf k}

\newcommand{\bL}{\mathbf L}
\newcommand{\bR}{\mathbf R}
\newcommand{\bC}{C}

\newcommand{\bi}{\mathbf i}
\newcommand{\bj}{\mathbf j}

\newcommand{\apo}{\sigma_{ap}^S}
\newcommand{\vr}{V_\quat^R}

\newcommand{\ra}{\text{ran}}
\newcommand{\kr}{\text{ker}}

\newcommand{\mfp}{\mathfrak{P}}
\newcommand{\gli}{\text{glim}}
\newcommand{\Aci}{A^\circ}
\newcommand{\fp}{\mathfrak{P}}

\newcommand{\sus}{\sigma_{su}^S}
\newcommand{\Bc}{B^\circ}
\newcommand{\Ihi}{\mathbb{I}_{\mathfrak{H}}}
\newcommand{\Ss}{\sigma_S}
\newcommand{\re}{\text{Re}}
\newcommand{\opu}{\overline{\mathfrak{p}}}

\begin{document}
\title[Berberian extension]{Berberian extension and its S-spectra in a quaternionic Hilbert space}
\author{B. Muraleetharan$^{\dagger}$ and 
K. Thirulogasanthar$^{\ddagger}$}
\address{$^{\dagger}$ Department of mathematics and Statistics, University of Jaffna, Thirunelveli, Sri Lanka.}
\address{$^{\ddagger}$ Department of Computer Science and Software Engineering, Concordia University, 1455 De Maisonneuve Blvd. West, Montreal, Quebec, H3G 1M8, Canada.}
\email{muralee@univ.jfn.ac.lk and santhar@gmail.com }
\date{\today}
\date{\today}
\begin{abstract}
For a bounded right linear operators $A$, in a right quaternionic Hilbert space $\vr$, following the complex formalism, we study the Berberian extension $A^\circ$, which is an extension of $A$ in a right quaternionic Hilbert space obtained from $\vr$. In the complex setting, the important feature of the Berberian extension is that it converts approximate point spectrum of $A$ into point spectrum of $A^\circ$. We show that the same is true for the quaternionic S-spectrum. As in the complex case, we use the Berberian extension to study some properties of the commutator of two quaternionic bounded right linear operators.  
\end{abstract}
\maketitle
\pagestyle{myheadings}
\section{Introduction}
In 1962 Berberian extended a bounded linear operator $A$ on a complex Hilbert space $X$ to an operator $A^\circ$ on a complex Hilbert space obtained from $X$. An important feature of this extension is that it converts approximate point spectrum of $A$ into point spectrum of $A^\circ$ \cite{Ber}. This extension is also a useful tool in studying the spectrum of commutator of two bounded linear operators \cite{La}.\\

In the complex theory this extension goes as follows. Let $X$ be a complex Hilbert space. Let $ l^\infty(X)$ denotes the space of all bounded sequence of elements of $X$, and let $c_0(X)$ denote the space of all null sequences in $X$. Endowed with the canonical norm, the space $\mathfrak{X}=l^\infty(X)/c_0(X)$ is a Hilbert space into which $X$ can be isometrically embedded. Every operator $A\in B(X)$, the set of all bounded linear operators on $X$, defines by component wise action an operator on $l^\infty(X)$ which leaves $c_0(X)$ invariant, and hence induces an operator $A^\circ\in B(\mathfrak{X})$. It is immediate that $A^\circ$ is an extension of $A$ when $X$ is regarded as a subspace of $\mathfrak{X}$, and that the mapping that assigns to each $A\in B(X)$ its Berberian extension $A^\circ\in B(\mathfrak{X})$ is an isometric algebra homomorphism.\\

In this note we shall study the Berberian extension of a quaternionic right linear operator $A$ on a right quaternionic Hilbert space and show that the approximate point S-spectrum of $A$ coincides with the point S-spectrum of the Berberian extension $A^\circ$. Following the complex formalism given in \cite{La}, we shall also study certain S-spectral properties of the commutator of two quaternionic bounded right linear operators.\\ 

In the complex setting, in a complex Hilbert space or Banach space $\HI$, for a bounded linear operator, $A$, the spectrum is defined as the set of complex numbers $\lambda$ for which the operator  $Q_\lambda(A)=A-\lambda \mathbb{I}_\HI$, where $\mathbb{I}_{\HI}$ is the identity operator on $\HI$,  is not invertible. In the quaternionic setting, let $\vr$ be a separable right quaternionic Hilbert space or Banach space,  $A$ be a bounded right linear operator, and $R_\qu(A)=A^2-2\text{Re}(\qu)A+|\qu|^2\Iop$, with $\qu\in\quat$, the set of all quaternions, be the pseudo-resolvent operator. The S-spectrum is defined as the set of quaternions $\qu$ for which $R_\qu(A)$ is not invertible. In the complex case various classes of spectra, such as approximate point spectrum, surjectivity spectrum etc. are defined by placing restrictions on the operator $Q_\lambda(A)$. In this regard, in the quaternionic setting, these spectra are also defined by placing the same restrictions to the operator $R_\qu(A)$ \cite{Fr, Kato}.\\

Due to the non-commutativity, in the quaternionic case  there are three types of  Hilbert spaces: left, right, and two-sided, depending on how vectors are multiplied by scalars. This fact can entail several problems. For example, when a Hilbert space $\mathcal H$ is one-sided (either left or right) the set of linear operators acting on it does not have a linear structure. Moreover, in a one sided quaternionic Hilbert space, given a linear operator $A$ and a quaternion $\mathfrak{q}\in\quat$, in general we have that $(\mathfrak{q} A)^{\dagger}\not=\overline{\mathfrak{q}} A^{\dagger}$ (see \cite{Mu} for details). These restrictions can severely prevent the generalization  to the quaternionic case of results valid in the complex setting. Even though most of the linear spaces are one-sided, it is possible to introduce a notion of multiplication on both sides by fixing an arbitrary Hilbert basis of $\mathcal H$.  This fact allows to have a linear structure on the set of linear operators, which is a minimal requirement to develop a full theory.

\section{Mathematical preliminaries}
In order to make the paper self-contained, we recall some facts about quaternions which may not be well-known.  For details we refer the reader to \cite{Ad,ghimorper,Vis}.
\subsection{Quaternions}
Let $\quat$ denote the field of all quaternions and $\quat^*$ the group (under quaternionic multiplication) of all invertible quaternions. A general quaternion can be written as
$$\qu = q_0 + q_1 \bi + q_2 \bj + q_3 \bk, \qquad q_0 , q_1, q_2, q_3 \in \mathbb R, $$
where $\bi,\bj,\bk$ are the three quaternionic imaginary units, satisfying
$\bi^2 = \bj^2 = \bk^2 = -1$ and $\bi\bj = \bk = -\bj\bi,  \; \bj\bk = \bi = -\bk\bj,
\; \bk\bi = \bj = - \bi\bk$. The quaternionic conjugate of $\qu$ is
$$ \overline{\qu} = q_0 - \bi q_1 - \bj q_2 - \bk q_3 , $$
while $\vert \qu \vert=(\qu \overline{\qu})^{1/2} $ denotes the usual norm of the quaternion $\qu$.
If $\qu$ is non-zero element, it has inverse
$
\qu^{-1} =  \dfrac {\overline{\qu}}{\vert \qu \vert^2 }.$

\subsection{Quaternionic Hilbert spaces}
In this subsection we  discuss right quaternionic Hilbert spaces. For more details we refer the reader to \cite{Ad,ghimorper,Vis}.
\subsubsection{Right quaternionic Hilbert Space}
Let $V_{\quat}^{R}$ be a vector space under right multiplication by quaternions.  For $\phi,\psi,\omega\in V_{\quat}^{R}$ and $\qu\in \quat$, the inner product
$$\langle\cdot\mid\cdot\rangle_{V_{\quat}^{R}}:V_{\quat}^{R}\times V_{\quat}^{R}\longrightarrow \quat$$
satisfies the following properties
\begin{enumerate}
	\item[(i)]
	$\overline{\langle \phi\mid \psi\rangle_{V_{\quat}^{R}}}=\langle \psi\mid \phi\rangle_{V_{\quat}^{R}}$
	\item[(ii)]
	$\|\phi\|^{2}_{V_{\quat}^{R}}=\langle \phi\mid \phi\rangle_{V_{\quat}^{R}}>0$ unless $\phi=0$, a real norm
	\item[(iii)]
	$\langle \phi\mid \psi+\omega\rangle_{V_{\quat}^{R}}=\langle \phi\mid \psi\rangle_{V_{\quat}^{R}}+\langle \phi\mid \omega\rangle_{V_{\quat}^{R}}$
	\item[(iv)]
	$\langle \phi\mid \psi\qu\rangle_{V_{\quat}^{R}}=\langle \phi\mid \psi\rangle_{V_{\quat}^{R}}\qu$
	\item[(v)]
	$\langle \phi\qu\mid \psi\rangle_{V_{\quat}^{R}}=\overline{\qu}\langle \phi\mid \psi\rangle_{V_{\quat}^{R}}$
\end{enumerate}
where $\overline{\qu}$ stands for the quaternionic conjugate. It is always assumed that the
space $V_{\quat}^{R}$ is complete under the norm given above and separable. Then,  together with $\langle\cdot\mid\cdot\rangle_{\vr}$ this defines a right quaternionic Hilbert space. Quaternionic Hilbert spaces share many of the standard properties of complex Hilbert spaces. Every separable quaternionic Hilbert space posses a basis. It should be noted that once a Hilbert basis is fixed, every left (resp. right) quaternionic Hilbert space also becomes a right (resp. left) quaternionic Hilbert space \cite{ghimorper,Vis}.

The field of quaternions $\quat$ itself can be turned into a left quaternionic Hilbert space by defining the inner product $\langle \qu \mid \qu^\prime \rangle = \qu \overline{\qu^{\prime}}$ or into a right quaternionic Hilbert space with  $\langle \qu \mid \qu^\prime \rangle = \overline{\qu}\qu^\prime$.
\section{Right quaternionic linear  operators and some basic properties}
In this section we shall define right  $\quat$-linear operators and recall some basis properties. Most of them are very well known. In this manuscript, we follow the notations in \cite{AC} and \cite{ghimorper}. We shall also recall some results pertinent to the development of the paper. 
\begin{definition}
A mapping $A:\D(A)\subseteq V_{\quat}^R \longrightarrow V_{\quat}^R$, where $\D(A)$ stands for the domain of $A$, is said to be right $\quat$-linear operator or, for simplicity, right linear operator, if
$$A(\phi\bfraka+\psi\bfrakb)=(A\phi)\bfraka+(A\psi)\bfrakb,~~\mbox{~if~}~~\phi,\,\psi\in \D(A)~~\mbox{~and~}~~\bfraka,\bfrakb\in\quat.$$
\end{definition}
The set of all right linear operators from $V_{\quat}^{R}$ to $V_{\quat}^{R}$ will be denoted by $\mathcal{L}(V_{\quat}^{R})$ and the identity linear operator on $V_{\quat}^{R}$ will be denoted by $\Iop$. For a given $A\in \mathcal{L}(V_{\quat}^{R})$, the range and the kernel will be
\begin{eqnarray*}
\text{ran}(A)&=&\{\psi \in V_{\quat}^{R}~|~A\phi =\psi \quad\text{for}~~\phi \in\D(A)\}\\
\ker(A)&=&\{\phi \in\D(A)~|~A\phi =0\}.
\end{eqnarray*}
We call an operator $A\in \mathcal{L}(V_{\quat}^{R})$ bounded if
\begin{equation}\label{PE1}
\|A\|=\sup_{\|\phi \|_{\vr}=1}\|A\phi \|_{\vr}<\infty,
\end{equation}
or equivalently, there exist $K\geq 0$ such that $\|A\phi \|_{\vr}\leq K\|\phi \|_{\vr}$ for all $\phi \in\D(A)$. The set of all bounded right linear operators from $V_{\quat}^{R}$ to $V_{\quat}^{R}$ will be denoted by $B(V_{\quat}^{R})$. 
\\
Assume that $V_{\quat}^{R}$ is a right quaternionic Hilbert space, $A$ is a right linear operator acting on it.
Then, there exists a unique linear operator $A^{\dagger}$ such that
\begin{equation}\label{Ad1}
\langle \psi \mid A\phi \rangle_{\vr}=\langle A^{\dagger} \psi \mid\phi \rangle_{\vr};\quad\text{for all}~~~\phi \in \D (A), \psi\in\D(A^\dagger),
\end{equation}
where the domain $\D(A^\dagger)$ of $A^\dagger$ is defined by
$$
\D(A^\dagger)=\{\psi\in V_{\quat}^{R}\ |\ \exists \varphi\ {\rm such\ that\ } \langle \psi \mid A\phi \rangle_{\vr}=\langle \varphi \mid\phi \rangle_{\vr}\}.$$
\begin{proposition}\cite{ghimorper}\label{da}
	Let $A, B\in B(\vr)$ then
	\begin{enumerate}
		\item [(a)] $(A+B)^\dagger=A^\dagger+B^\dagger$.
		\item[(b)] $(AB)^\dagger=B^\dagger A^\dagger$.
	\end{enumerate}
\end{proposition}
We shall need the following results which are already appeared in \cite{ghimorper,Fr}.
\begin{proposition}\label{IP30}
Let $A\in B(\vr)$. Then
\begin{itemize}
\item [(a)] $\ra(A)^\perp=\kr(A^\dagger).$
\item [(b)] $\kr(A)=\ra(A^\dagger)^\perp.$
\item [(c)] $\kr(A)$ is closed subspace of $\vr$.
\end{itemize}
\end{proposition}
\begin{theorem}\cite{Fr}(Bounded inverse theorem)\label{NT1}
Let $A\in B(\vr)$, then the following results are equivalent.
\begin{enumerate}
\item [(a)] $A$ has a bounded inverse on its range.
\item[(b)] $A$ is bounded below.
\item[(c)] $A$ is injective and has a closed range.
\end{enumerate}
\end{theorem}
\begin{proposition}\cite{Fr}\label{NP2}
Let $A\in\B(\vr)$. Then,
\begin{enumerate}
\item [(a)] $A$ is invertible if and only if it is injective with a closed range (i.e., $\kr(A)=\{0\}$ and $\overline{\ra(A)}=\ra(A)$).
\item[(b)] $A$ is left (right) invertible if and only if $A^\dagger$ is right (left) invertible.
\end{enumerate}
\end{proposition}
\begin{proposition}\cite{Fr}\label{PS}
	$A\in B(\vr)$ is surjective if and only if $A$ is right invertible.
\end{proposition}
\begin{proposition}\label{PI}
	$A\in B(\vr)$ is injective if and only if $A$ is left invertible.
\end{proposition}
\begin{proof}
From point (b) of proposition \ref{NP2}, point (b) of proposition \ref{IP30}, and proposition \ref{PS}, we have,
$A$ is left invertible $\Leftrightarrow$ $A^\dagger$ is right invertible $\Leftrightarrow$ $\ra(A^\dagger)=\vr$ $\Leftrightarrow$ $\kr(A)=\{0\}.$ This completes the proof.
\end{proof}
\subsection{Left Scalar Multiplications on $\vr$.}
We shall extract the definition and some properties of left scalar multiples of vectors on $\vr$ from \cite{ghimorper} as needed for the development of the manuscript. The left scalar multiple of vectors on a right quaternionic Hilbert space is an extremely non-canonical operation associated with a choice of preferred Hilbert basis. Since $\vr$ is a separable Hilbert space, $\vr$ has a Hilbert basis
\begin{equation}\label{b1}
\mathcal{O}=\{\varphi_{k}\,\mid\,k\in N\},
\end{equation}
where $N$ is a countable index set.
The left scalar multiplication on $\vr$ induced by $\mathcal{O}$ is defined as the map $\quat\times \vr\ni(\qu,\phi)\longmapsto \qu\phi\in \vr$ given by
\begin{equation}\label{LPro}
\qu\phi:=\sum_{k\in N}\varphi_{k}\qu\langle \varphi_{k}\mid \phi\rangle_{\vr},
\end{equation}
for all $(\qu,\phi)\in\quat\times \vr$.
\begin{proposition}\cite{ghimorper}\label{lft_mul}
	The left product defined in the equation \ref{LPro} satisfies the following properties. For every $\phi,\psi\in \vr$ and $\bfrakp,\qu\in\quat$,
	\begin{itemize}
		\item[(a)] $\qu(\phi+\psi)=\qu\phi+\qu\psi$ and $\qu(\phi\bfrakp)=(\qu\phi)\bfrakp$.
		\item[(b)] $\|\qu\phi\|_{\vr}=|\qu|\|\phi\|_{\vr}$.
		\item[(c)] $\qu(\bfrakp\phi)=(\qu\bfrakp)\phi$.
		\item[(d)] $\langle\overline{\qu}\phi\mid\psi\rangle_{\vr}
		=\langle\phi\mid\qu\psi\rangle_{\vr}$.
		\item[(e)] $r\phi=\phi r$, for all $r\in \mathbb{R}$.
		\item[(f)] $\qu\varphi_{k}=\varphi_{k}\qu$, for all $k\in N$.
	\end{itemize}
\end{proposition}
\begin{remark}
	(1) The meaning of writing $\bfrakp\phi$ is $\bfrakp\cdot\phi$, because the notation from the equation \ref{LPro} may be confusing, when $\vr=\quat$. However, regarding the field $\quat$ itself as a right $\quat$-Hilbert space, an orthonormal basis $\mathcal{O}$ should consist only of a singleton, say $\{\varphi_{0}\}$, with $\mid\varphi_{0}\mid=1$, because we clearly have $\theta=\varphi_{0}\langle\varphi_{0}\mid\theta\rangle$, for all $\theta\in\quat$. The equality from (f) of proposition \ref{lft_mul} can be written as $\bfrakp\varphi_{0}=\varphi_{0}\bfrakp$, for all $\bfrakp\in\quat$. In fact, the left hand may be confusing and it should be understood as $\bfrakp\cdot\varphi_{0}$, because the true equality $\bfrakp\varphi_{0}=\varphi_{0}\bfrakp$ would imply that $\varphi_{0}=\pm 1$. For the simplicity, we are writing $\bfrakp\phi$ instead of writing $\bfrakp\cdot\phi$.\\
	(2) Also one can trivially see that $(\bfrakp+\qu)\phi=\bfrakp\phi+\qu\phi$, for all $\bfrakp,\qu\in\quat$ and $\phi\in \vr$.
\end{remark}
Furthermore, the quaternionic left scalar multiplication of linear operators is also defined in \cite{Fab1}, \cite{ghimorper}. For any fixed $\qu\in\quat$ and a given right linear operator $A:\vr\longrightarrow \vr$, the left scalar multiplication of $A$ is defined as a map $\qu A:\vr\longrightarrow \vr$ by the setting
\begin{equation}\label{lft_mul-op}
(\qu A)\phi:=\qu (A\phi)=\sum_{k\in N}\varphi_{k}\qu\langle \varphi_{k}\mid A\phi\rangle_{\vr},
\end{equation}
for all $\phi\in \vr$. It is straightforward that $\qu A$ is a right linear operator. We can define right scalar multiplication of the right linear operator $A:\vr\longrightarrow \vr$ as a map $ A\qu:\vr\longrightarrow \vr$ by the setting
\begin{equation}\label{rgt_mul-op}
(A\qu)\phi:=A(\qu \phi),
\end{equation}
for all $\phi\in \vr$. It is also right linear operator. One can easily see that
\begin{equation}\label{sc_mul_aj-op}
(\qu A)^{\dagger}=A^{\dagger}\overline{\qu}~\mbox{~and~}~
(A\qu)^{\dagger}=\overline{\qu}A^{\dagger}.
\end{equation}
\subsection{S-Spectrum}
For a given right linear operator $A:\D(A)\subseteq V_{\quat}^R\longrightarrow V_{\quat}^R$ and $\qu\in\quat$, we define the operator $R_{\qu}(A):\D(A^{2})\longrightarrow\quat$ by  $$R_{\qu}(A)=A^{2}-2\text{Re}(\qu)A+|\qu|^{2}\Iop,$$
where $\qu=q_{0}+\bi q_1 + \bj q_2 + \bk q_3$ is a quaternion, $\text{Re}(\qu)=q_{0}$  and $|\qu|^{2}=q_{0}^{2}+q_{1}^{2}+q_{2}^{2}+q_{3}^{2}.$\\
In the literature, the operator is called pseudo-resolvent since it is not the resolvent operator of $A$ but it is the one related to the notion of spectrum as we shall see in the next definition. For more information, on the notion of $S$-spectrum the reader may consult e.g. \cite{Fab, Fab1, NFC}, and  \cite{ghimorper}.
\begin{definition}
	Let $A:\D(A)\subseteq V_{\quat}^R\longrightarrow V_{\quat}^R$ be a right linear operator. The {\em $S$-resolvent set} (also called \textit{spherical resolvent} set) of $A$ is the set $\rho_{S}(A)\,(\subset\quat)$ such that the three following conditions hold true:
	\begin{itemize}
		\item[(a)] $\ker(R_{\qu}(A))=\{0\}$.
		\item[(b)] $\text{ran}(R_{\qu}(A))$ is dense in $V_{\quat}^{R}$.
		\item[(c)] $R_{\qu}(A)^{-1}:\text{ran}(R_{\qu}(A))\longrightarrow\D(A^{2})$ is bounded.
	\end{itemize}
	The \textit{$S$-spectrum} (also called \textit{spherical spectrum}) $\sigma_{S}(A)$ of $A$ is defined by setting $\sigma_{S}(A):=\quat\smallsetminus\rho_{S}(A)$. For a bounded linear operator $A$ we can write the resolvent set as
	\begin{eqnarray*}
		\rho_S(A)
		&=&\{\qu\in\quat~|~R_\qu(A)~\text{has an inverse in}~B(V_{\quat}^R)\}\\
		&=&\{\qu\in\quat~|~\text{ker}(R_\qu(A))=\{0\}\quad\text{and}\quad \text{ran}(R_\qu(A))=V_\quat^R\}
	\end{eqnarray*}
	and the spectrum can be written as
	\begin{eqnarray*}
		\sigma_S(A)&=&\quat\setminus\rho_S(A)\\
		&=&\{\qu\in\quat~|~R_\qu(A)~\text{has no inverse in}~B(V_{\quat}^R)\}\\
		&=&\{\qu\in\quat~|~\text{ker}(R_\qu(A))\not=\{0\}\quad\text{or}\quad \text{ran}(R_\qu(A))\not=V_\quat^R\}.
	\end{eqnarray*}
	The spectrum $\sigma_S(A)$ decomposes into three major disjoint subsets as follows:
	\begin{itemize}
		\item[(i)] the \textit{spherical point spectrum} of $A$: $$\sigma_{pS}(A):=\{\qu\in\quat~\mid~\ker(R_{\qu}(A))\ne\{0\}\}.$$
		\item[(ii)] the \textit{spherical residual spectrum} of $A$: $$\sigma_{rS}(A):=\{\qu\in\quat~\mid~\ker(R_{\qu}(A))=\{0\},\overline{\text{ran}(R_{\qu}(A))}\ne V_{\quat}^{R}~\}.$$
		\item[(iii)] the \textit{spherical continuous spectrum} of $A$: $$\sigma_{cS}(A):=\{\qu\in\quat~\mid~\ker(R_{\qu}(A))=\{0\},\overline{\text{ran}(R_{\qu}(A))}= V_{\quat}^{R}, R_{\qu}(A)^{-1}\notin B(V_{\quat}^{R}) ~\}.$$
	\end{itemize}
	If $A\phi=\phi\qu$ for some $\qu\in\quat$ and $\phi\in V_{\quat}^{R}\smallsetminus\{0\}$, then $\phi$ is called an \textit{eigenvector of $A$ with right eigenvalue} $\qu$. The set of right eigenvalues coincides with the point $S$-spectrum, see \cite{ghimorper}, proposition 4.5.
\end{definition}
\begin{proposition}\cite{Fab2, ghimorper}\label{PP1}
	For $A\in B(\vr)$, the resolvent set $\rho_S(A)$ is a non-empty open set and the spectrum $\sigma_S(A)$ is a non-empty compact set.
\end{proposition}
\begin{remark}\label{R1}
	For $A\in B(\vr)$, since $\sigma_S(A)$ is a non-empty compact set so is its boundary. That is, $\partial\sigma_S(A)=\partial\rho_S(A)\not=\emptyset$.
\end{remark}
\begin{proposition}\cite{Jo}\label{Jo} Let $A\in B(\vr)$. Then $\kr(R_\qu(A))\not=\{0\}$ if and only if $\qu$ is a right eigenvalue of $A$. In particular every right eigenvalue belongs to $\sigma_S(A)$.
\end{proposition}
\begin{definition}\cite{Fr}\label{D1}
	Let $A\in B(V_\quat^R)$. The {\em approximate S-point spectrum} of $A$, denoted by $\apo(A)$, is defined as
	$$\apo(A)=\{\qu\in\quat~~|~~\text{there is a sequence}~~\{\phi_n\}_{n=1}^{\infty}~~\text{such that}~~\|\phi_n\|=1~~\text{and}~~\|R_\qu(A)\phi_n\|\longrightarrow 0\}.$$
\end{definition}
\begin{proposition}\label{P3}\cite{Fr}
	Let $A\in B(\vr)$, then $\sigma_{pS}(A)\subseteq\apo(A)$. 
\end{proposition}
\begin{definition}\cite{Fr,Kato}\label{DC}
	The spherical compression spectrum of an operator $A\in B(\vr)$, denoted by $\sigma_c^S(A)$, is defined as
	$$\sigma_c^S(A)=\{\qu\in\quat~~|~~\text{ran}(R_\qu(A))~~~\text{is not dense in}~~\vr~\}.$$
\end{definition}
\begin{definition}\cite{Kato}\label{su} Let $A\in B(\vr)$. The surjectivity S-spectrum of $A$ is defined as
	$$\sus(A)=\{\qu\in\quat~|~\text{ran}(R_{\qu}(A)\not=\vr\}.$$
\end{definition}
Clearly we have
\begin{equation}\label{sue1}
\sigma_c^S(A)\subseteq\sus(A)\quad\text{and}\quad\sigma_S(A)=\sigma_{pS}(A)\cup\sus(A).
\end{equation}
\begin{proposition}\label{su1}\cite{Fr} Let $A\in B(\vr)$. Then $A$ has the following properties.
	\begin{enumerate}
		\item[(a)]$\sigma_{pS}(A)\subseteq\sigma_c^S(A^\dagger)~~\text{and}~~\sigma_c^S(A)=\sigma_{pS}(A^\dagger)$.
		\item[(b)]$\sus(A)=\apo(A^\dagger)~~\text{and}~~\apo(A)=\sus(A^\dagger).$
		\item[(c)]$\sigma_S(A)=\sigma_S(A^\dagger).$
	\end{enumerate}
\end{proposition}
\begin{proposition}\label{P4}\cite{Fr}
	If $A\in B(\vr)$ and $\qu\in\quat$, then the following statements are equivalent.
	\begin{enumerate}
		\item[(a)] $\qu\not\in\apo(A).$
		\item[(b)] $\text{ker}(R_\qu(A))=\{0\}$ and $\text{ran}(R_\qu(A))$ is closed.
		\item[(c)] There exists a constant $c\in\R$, $c>0$ such that $\|R_\qu(A)\phi\|\geq c\|\phi\|$ for all $\phi\in\D(A^2)$.
	\end{enumerate}
\end{proposition}
\begin{theorem}\cite{ghimorper}\label{ET0}
	Let $\vr$ be a right quaternionic Hilbert space equipped with a left scalar multiplication. Then the set $B(\vr)$ equipped with the point-wise sum, with the left and right scalar multiplications defined in equations \ref{lft_mul-op} and \ref{rgt_mul-op}, with the composition as product, with the adjunction $A\longrightarrow A^\dagger$, as in \ref{Ad1}, as $^*-$ involution and with the norm defined in \ref{PE1}, is a quaternionic two-sided Banach $C^*$-algebra with unity $\Iop$.
\end{theorem}
	One can observe that in the above theorem, if the left scalar multiplication is left out on $\vr$, then $B(\vr)$ becomes a real Banach $C^*$-algebra with unity $\Iop$.
\section{Berberian extension in the quaternionic setting}
Following the definition given in \cite{Ber} for complex bounded sequences, we denote by $\gli$ a {\em Banach generalized limit}  defined for bounded sequences $\{\qu_n\}\subseteq\quat$ with the following properties. For $\qu\in\quat$ and $\{\qu_n\}, \{\pu_n\}\subseteq\quat$,
\begin{enumerate}
	\item [(a)] $\gli(\qu_n+\pu_n)=\gli(\qu_n)+\gli(\pu_n)$;
	\item[(b)] $\gli(\qu_n\qu)=\gli(\qu_n)\,\qu$;
	\item[(c)] $\gli(\qu\qu_n)=\qu\,\gli(\qu_n)$;
	\item[(d)] $\displaystyle\gli(\qu_n)=\lim_{n\rightarrow\infty}\qu_n$ whenever $\{\qu_n\}$ is convergent;
	\item[(e)] $\gli(\qu_n)\geq 0$ when $\{\qu_n\}\subseteq\R$ and $\qu_n\geq 0$ for all $n$.
\end{enumerate}
$\gli$ defines a positive linear form on the vector space $\M$ of all quaternionic bounded sequences and  $c_0$ denotes the set of quaternionic null sequences, that is, sequences that converge to zero, and has the value $1$ for the constant sequence $\{1\}$. From properties (a) and (e) of $\gli$, $\gli(\qu_n)$ is real whenever $\qu_n$ is real for all $n$. Hence $\gli(\oqu_n)=\overline{\gli(\qu_n)}$ for any bounded sequence $\{\qu_n\}\subseteq\quat$.
\subsection{An extension of $\vr$}~
Let
$$\B=\left\{s=\{\phi_n\}~\vert~ \{\phi_n\}\subseteq\vr,~~\|\phi_n\|_{\vr}<\infty~\forall n,~\text{that is,}~~\{\|\phi_n\|_{\vr}\}\in\M\right\}.$$
If $s=\{\phi_n\}$ and $t=\{\psi_n\}$ write $s=t$ whenever $\phi_n=\psi_n$ for all $n$. Also
$$s+t=\{\phi_n+\psi_n\}\quad\text{and}\quad s\qu=\{\phi_n\qu\},$$
with these operations $\B$ becomes a quaternionic right linear vector space.
The left scalar multiplication, on $\B$, is defined as the map $\quat\times \B\ni(\qu,s)\longmapsto \qu s\in \B$ given by
\begin{equation}\label{LProB}
	\qu s:=\{\q\phi_n\},
\end{equation}
for all $(\qu,s)=(\qu,\{\phi_n\})\in\quat\times \B$, where for each $n\in\mathbb{N}, \q\phi_n$ is given by the definition \ref{LPro}.
Suppose that $s=\{\phi_n\}, t=\{\psi_n\}\in\B$. Since
$$|\langle\phi_n|\psi_n\rangle_{\vr}|\leq\|\phi_n\|_{\vr}~\|\psi_n\|_{\vr},\quad\text{for all}~n,$$
it is permissible to define
$$\Phi(s,t)=\gli(\langle\phi_n|\psi_n\rangle_{\vr}).$$
We have the following properties for $\Phi$.
\begin{enumerate}
	\item [(a)] Since $\langle\phi_n|\psi_n\rangle_{\vr}=\overline{\langle\psi_n|\phi_n\rangle_{\vr}},$ we have $\Phi(s,t)=\overline{\Phi(t,s)}$. That is, $\Phi$ is symmetric.
	\item[(b)] Since $\langle\phi_n|\phi_n\rangle_{\vr}\geq 0$ for all $n$, $\Phi(s,s)\geq 0$ for all $s\in\B$. That is, $\Phi$ is positive.
	\item[(c)] $\Phi$ is a bilinear functional, in the sense that $\Phi$ is left-antilinear with respect to the first variable,
	$$\Phi(r\pu+s\qu,t)=\overline{\pu}\Phi(r,t)+\overline{\qu}\Phi(s,t),\quad\text{for all}~\pu,\qu\in\quat\text{~~and~~}r,s,t\in\B,$$
	and $\Phi$ is right-linear with respect to the second variable,
	 $$\Phi(s,r\pu+t\qu)=\Phi(s,r)\pu+\Phi(s,t)\qu,\quad\text{for all}~\pu,\qu\in\quat\text{~~and~~}r,s,t\in\B.$$  
\end{enumerate}
From the Schwarz's inequality we have
$$|\Phi(s,t)|^2\leq \Phi(s,s)\Phi(t,t).$$
Let 
$$\mathfrak{N}=\{s\in\B~\vert~\Phi(s,s)=0\}=\{s\in\B~\vert~\Phi(s,t)=0~~\forall~t\in\B\}.$$
Clearly $\mathfrak{N}$ is a right linear subspace of $\B$. Write $[s]=s+\mathfrak{N}$ for a coset. The quotient right linear vector space $\mathfrak{P}=\B/\mathfrak{N}$ becomes an inner product space by defining
	$$\langle[s]~|~[t]\rangle_{\mfp}=\Phi(s,t).$$
If $u=\{[\phi_n]\}=\{\phi_n\}+\mathfrak{N}$ and $v=\{[\psi_n]\}=\{\psi_n\}+\mathfrak{N}$, then 
\begin{equation}\label{inneq}
\langle u~|~v\rangle_{\mfp}=\langle [\phi_n]~|~[\psi_n]\rangle_{\mfp}=\gli\langle\phi_n~|~\psi_n\rangle_{\vr}.
\end{equation}
Using the left scalar multiplication defined on $\B$, by the equation \ref{LProB}, we can define a left scalar multiplication on $\mathfrak{P}$ by the map $\quat\times \mathfrak{P}\ni(\qu,s)\longmapsto \qu [s]\in \mathfrak{P}$ given by
\begin{equation}\label{LProBN}
\qu [s]:=\qu s+\mathfrak{N},
\end{equation}
for all $(\qu,[s])=(\qu,s+\mathfrak{N})\in\quat\times \mathfrak{P}$. Following proposition provides some properties of the above defined left scalar multiplication:
\begin{proposition}\label{lft_mulBN}
	The left product defined in the equation \ref{LProBN} satisfies the following properties. For every $[s],[t]\in \mfp$ and $\bfrakp,\qu\in\quat$,
	\begin{itemize}
		\item[(a)] $\qu([s]+[t])=\qu[s]+\qu[t]$ and $\qu([s]\bfrakp)=(\qu[s])\bfrakp$.
		\item[(b)] $\|\,\qu[s]\,\|_{\mfp}=|\qu|\|\,[s]\,\|_{\mfp}$.
		\item[(c)] $\qu(\bfrakp[s])=(\qu\bfrakp)[s]$.
		\item[(d)] $\langle\overline{\qu}[s]\mid[t]\rangle_{\mfp}
		=\langle[s]\mid\qu[t]\rangle_{\mfp}$.
		\item[(e)] $r[s]=[s] r$, for all $r\in \mathbb{R}$.
	\end{itemize}
\end{proposition}
\begin{proof}
The proof immediately follows from the proposition \ref{lft_mul} together with the equations \ref{LProB} and \ref{LProBN}.
\end{proof}
Let $\phi\in\vr$, we write $\{\phi\}$ for the sequence all of whose terms are $\phi$ and $\phi'$ for the coset $\{[\phi]\}=\{\phi\}+\mathfrak{N}$. Evidently
$$\langle[\phi]~|~[\psi]\rangle_{\mfp}=\langle\phi|\psi\rangle_{\vr},$$
and $\phi\mapsto[\phi]$ is an isometric right linear mapping of $\vr$ onto a closed linear subspace ${\vr}'$ of $\mathfrak{P}$. Regard $\mathfrak{P}$ as a linear subspace of its Hilbert space completion $\HI$. Then ${\vr}'$ is a closed linear subspace of $\HI$ and $\mathfrak{P}$ is a dense linear subspace of $\HI$.
\subsection{A representation of $B(\vr)$} Every operator $A$ in $\vr$ determines an operator $A^\circ$ in $\HI$ as follows.\\
If $s=\{\phi_n\}\in\B$ then the relation $\|A\phi_n\|_{\vr}\leq\|A\|~\|\phi_n\|_{\vr}$ shows that $\{A\phi_n\}\in\B$. Define
$$A_0:\B\longrightarrow\B\quad \text{by}\quad A_0s=\{A\phi_n\},$$
then $A_0$ is a right linear mapping such that
$$\Phi(A_0s,A_0s)\leq \|A\|\Phi(s,s).$$
In particular, if $s\in\mathfrak{N}$, that is $\Phi(s,s)=0$, then $A_0s\in\mathfrak{N}$. it follows that
\begin{equation}\label{ex}
A^{\circ}:\mathfrak{P}\longrightarrow\mathfrak{P}\quad \text{by}\quad \{[\phi_n]\}\mapsto\{[A\phi_n]\}
\end{equation}
is a well-defined right linear map. Thus 
$$A^\circ s'=(A_0s)'$$
and the inequality
$$\langle A^\circ u|A^\circ u\rangle_{\mfp}\leq\|A\|^2\langle u|u\rangle_{\mfp}$$
is valid for all $u\in\mathfrak{P}$. That is, $\|A^\circ u\|_{\mfp}\leq\|A\|\|u\|_{\mfp}$, for all $u\in\mathfrak{P}$. Hence $A^\circ$ is bounded (continuous), and $\|A^\circ\|_\circ\leq \|A\|$, $\,\|\cdot\|_\circ$ is the norm on $B(\HI)$.
The left scalar multiplication of $A^\circ$ by any $\qu\in\quat$ is defined as a map $\qu A^\circ:\mfp\longrightarrow \mfp$ by the setting
\begin{equation}\label{lft_mul-opAcir}
(\qu A^\circ)\{[\phi_n]\}:=\{[\qu (A\phi_n)]\},
\end{equation}
for all $\{[\phi_n]\}\in \mfp$. It is straightforward that $\qu A^\circ$ is a right linear operator.
We also have the following properties for the operators: 
\begin{proposition}
	\label{C1}
	For $A$,$B\in B(\vr)$ and $\qu\in\quat$, we have
	\begin{enumerate}
		\item [(a)] $(A+B)^\circ=\Aci+B^\circ$,
		\item[(b)] $(\qu A)^\circ=\qu A^\circ$,
		\item[(c)] $(AB)^\circ=A^\circ B^\circ$,
		\item[(d)] $(A^\dagger)^\circ=(\Aci)^\dagger$,
		\item[(e)] $\Iop^\circ=\Iop$
		\item[(f)] $\|\Aci\|_\circ=\|A\|$.
	\end{enumerate}
\end{proposition}
\begin{proof}
Proofs of (a), (c) and (e) are straightforward from the definition of $A^\circ$. Assertion (b) immediately follows from the (definition) equation \ref{lft_mul-opAcir} as follows: for any $\{[\phi_n]\}\in\mfp$,
 $$(\qu A^\circ)\{[\phi_n]\}=\{[\qu (A\phi_n)]\}=\{[(\qu A)\phi_n]\}=(\qu A)^\circ\{[\phi_n]\}.$$ 
 To verify (d), let $C=(\Aci)^\dagger$ and $u=\{[\phi_n]\}$ and $v=\{[\psi_n]\}$. Then $$\langle \Aci u\mid v\rangle_\mfp=\langle u\mid Cv\rangle_{\mfp}.$$ This implies that
 \begin{eqnarray*}
 \langle u\mid Cv\rangle_{\mfp}=\langle \Aci u\mid v\rangle_\mfp=\gli(\langle A\phi_n|\psi_n\rangle_{\vr})
 =\gli(\langle \phi_n|A^\dagger\psi_n\rangle_{\vr})
 =\langle u\mid (A^\dagger)^\circ v\rangle_{\mfp}.
 \end{eqnarray*}
 Therefore $(A^\dagger)^\circ=C=(\Aci)^\dagger$, and this completes the proof of (d). Finally let us establish the equality $\|\Aci\|_\circ=\|A\|$. Firstly note that for any $\phi\in\vr$, from the equation \ref{inneq}, we have $\|\phi'\|_{\mfp}=\|\phi\|_{\vr}$. Now since $\Aci\phi'=(A\phi)'$, for all $\phi\in\vr$,
 $$\|\Aci\|_\circ=\sup_{\|\phi' \|_{\mfp}=1}\|\Aci\phi'\|_{\mfp}=\sup_{\|\phi \|_{\vr}=1}\|(A\phi)'\|_{\mfp}=\sup_{\|\phi \|_{\vr}=1}\|A\phi\|_{\vr}=\|A\|.$$
 Therefore the assertion (f) follows. 
\end{proof}
The continuous right linear mapping $\Aci$ extends to a unique right linear operator in $\HI$, which we also denote $\Aci$. Also in the theorem \ref{ET0}, $B(\vr)$ with left multiplication is a $C^*$-algebra with unity $\Iop$. In the same manner, $B(\HI)$ with left multiplication is a $C^*$-algebra with the same unity $\Iop$. Also note that, no matter which Hilbert basis we choose to define a left multiplication the spaces $B(\vr)$ and $B(\HI)$ becomes $C^*$-algebras, and hence the results provided in this note are independent of the basis chosen.
\begin{theorem}
The mapping
$$B(\vr)\longrightarrow B(\HI)\quad\text{by}\quad A\mapsto\Aci$$
is a faithful $*$-representation. 
\end{theorem}
\begin{proof}
The assertion (c) of proposition \ref{C1} verifies that the above map is a homomorphism. To check the injectivity of this map, suppose that $A,B\in B(\vr)$ with $A^\circ=B^\circ$. Then for any $\{[\phi_n]\}\in\mfp$, we have 
\begin{eqnarray*}
\{[A\phi_n]\}=\{[B\phi_n]\}\Rightarrow \{(A-B)\phi_n\}\in\mathfrak{N}
\,\Rightarrow\, \gli(\langle(A-B)\phi_n\mid(A-B)\phi_n\rangle_{V_{\quat}^{R}})=0. 
\end{eqnarray*}
Let $\phi\in\vr$, and choose $\phi_n=\phi,~\forall\,n\in\mathbb{N}$. Then $\|(A-B)\phi\|_{\vr}=0$. This concludes that $A=B$. Therefore the above map is injective. Hence the theorem follows. 
\end{proof}	
Suppose $A\geq 0$, that is $\langle A\phi|\phi\rangle_{\vr}\geq 0$ for all $\phi\in\vr$. If $u=\{\phi_n\}'\in\fp$, then $\langle A\phi_n|\phi_n\rangle_{\vr}\geq 0$ for all $n$, hence
$$\langle \Aci u|u\rangle_{\mfp}=\gli\langle A\phi_n|\phi_n\rangle_{\vr}\geq 0.$$ Hence $\langle \Aci v|v\rangle_{\mfp}\geq 0$ for all $v\in\HI$. Thus clearly for an operator $A$ in $\vr$ we have
\begin{equation}\label{EC1}
A\geq 0\Leftrightarrow \Aci\geq 0.
\end{equation}
\begin{proposition}\label{PC1}
		If $A\in B(\vr)$, then $\apo(\Aci)=\apo(A)$.
	\end{proposition}
	\begin{proof} Let $\qu\in\quat$. Then,
		$\qu\not\in\apo(A)$ if and only if there exists $\epsilon>0$ such that $R_\qu(A^\dagger)R_\qu(A)\geq\epsilon\Iop$. By equation \ref{EC1}, this condition is equivalent to  $R_\qu((\Aci)^\dagger)R_\qu(\Aci)\geq\epsilon\Iop$ thus $\qu\notin\apo(\Aci).$
	\end{proof}
	The following theorem is the key result of Berberian extension. 
	\begin{theorem}\label{TC1}
		For every operator $A\in B(\vr)$, we have $\apo(A)=\apo(\Aci)=\sigma_{pS}(\Aci).$
	\end{theorem}
	\begin{proof}
		From propositions \ref{P3}, \ref{PC1} the relation $\apo(A)=\apo(\Aci)\supseteq\sigma_{pS}(\Aci)$ is clear. Let $\qu\in\apo(A)$. Then there exists a sequence $\{\phi_n\}\subseteq\vr$ with $\|\phi_n\|_{\vr}=1$ such that $\|R_\qu(A)\phi_n\|_{\vr}\rightarrow 0$. Set $u=\{\phi_n\}'$, clearly $\|u\|_{\fp}=1$. Also
		$$\|R_\qu(\Aci)u\|_{\fp}=\gli\|R_\qu(A)\phi_n\|_{\vr}\rightarrow 0.$$
		Therefore, by proposition \ref{Jo}, $\qu$ is a right eigenvalue of $\Aci$. Hence $\qu\in\sigma_{pS}(\Aci)$, which completes the proof.
	\end{proof}
\section{Application to commutators in the quaternionic setting}
In the complex setting, the Berberian extension is very useful in studying spectral properties of commutators \cite{La}. Following the complex formalism, in this section, we shall study some properties of S-spectrum of commutators in the quaternionic setting.
\begin{proposition}\label{COP1}
	Let $A, B\in B(\vr)$ such that $AB=BA$, then
	\begin{enumerate}
		\item [(a)] $\apo(A+B)\subseteq\apo(A)+\apo(B)$,
		\item[(b)] $\sus(A+B)\subseteq\sus(A)+\sus(B)$,
		\item[(c)] $\sigma_S(A+B)\subseteq\sigma_S(A)+\sigma_S(B)$.
	\end{enumerate}
\end{proposition}
\begin{proof}
	(a)~Since $AB=BA$ we have $\Aci\Bc=\Bc\Aci$. Let $\qu\in\apo(A+B)=\sigma_{pS}(\Aci+\Bc)$. Let $Z=\kr(R_\qu(\Aci+\Bc))$. then $Z\not=\emptyset$. Let $\psi\in \Aci Z$, then $\psi=\Aci\phi$ for some $\phi\in Z$ and also $R_\qu(\Aci+\Bc)\phi=0$. Now
	$$R_\qu(\Aci+\Bc)\psi=R_\qu(\Aci+\Bc)\Aci\phi=\Aci R_\qu(\Aci+\Bc)\phi=0.$$
	Therefore $\psi\in Z$, hence $\Aci Z\subseteq Z$. That is, $Z$ is invariant under $\Aci$, and therefore $\apo(\Aci|Z)\not=\emptyset$. Let $\pu\in\apo(\Aci|Z)=\sigma_{pS}(\Aci|Z)$, hence $\Aci-\pu\Ihi=0$. Since $\qu\in\sigma_{pS}(\Aci+\Bc)$, we have $\Aci+\Bc-\qu\Ihi=0$, that is $\Bc=\qu\Ihi-\Aci$. Therefore,
	$$\Bc-(\qu-\pu)\Ihi=-(\Aci-\pu)\Ihi=0~~\text{on}~~Z.$$
	Thus $\qu-\pu\in \sigma_{pS}(\Bc|Z).$ Hence, from proposition \ref{PC1},
	$$\qu=\pu+(\qu-\pu)\in\sigma_{pS}(\Aci)+\sigma_{pS}(\Bc)
	=\apo(\Aci)+\apo(\Bc)=\apo(A)+\apo(B).$$ This completes the proof of (a).\\
	(b)~ Since $AB=BA$, we have $A^\dagger B^\dagger=B^\dagger A^\dagger$, and therefore (a) holds for $A^\dagger, B^\dagger$. Further from proposition \ref{su1}, part (b), $\sus(A)=\apo(A^\dagger)$. Thus (b) follows.\\
	(c)~ For any $A\in B(\vr)$, from equation \ref{sue1}, proposition \ref{P3}, we have $\sigma_S(A)=\sigma_{pS}(A)\cup\sus(A)\subseteq\apo(A)\cup\sus(A)$. And clearly $\apo(A),\sus(A)\subseteq\sigma_S(A)$. Therefore, from (a) and (b), we have
	$$\apo(A+B)\subseteq\apo(A)+\apo(B)\subseteq\sigma_S(A)+\sigma_S(B)$$
		and
	$$\sus(A+B)\subseteq\sus(A)+\sus(B)\subseteq\sigma_S(A)+\sigma_S(B).$$
Thus
	$$\sigma_S(A+B)\subseteq\apo(A+B)\cup\sus(A+B)\subseteq\sigma_S(A)+\sigma_S(B).$$
	Hence the inclusion (c) holds.
\end{proof}
\begin{definition}\label{CD1}
	Given $S, T\in B(\vr)$, the commutator $C(S,T):B(\vr)\longrightarrow B(\vr)$ is the mapping
	$$C(S,T)(A)=SA-AT=\bL_S(A)-\bR_T(A),\quad\text{for all}~~A\in B(\vr),$$
	where $\bL_S(A)=SA$ and $\bR_T(A)=AT$. It is clear that $A\in B(\vr)$ intertwines the pair $(S,T)$ precisely when $C(S,T)=0$.
\end{definition}
\begin{remark}
It is worth  noting the following results: for any $\qu\in\quat$ and $S,T\in B(\vr)$,
\begin{enumerate}
	\item $\bL_S\bR_T=\bR_T\bL_S$,
	\item $R_\qu(\bL_S)=R_\qu(S)$,
	\item $R_\qu(\bR_T)A=AR_\qu(T)$; for all $A\in B(\vr)$,
	\item $\bL_{R_{\qu}(S)}=R_\qu(\bL_S)$,
	\item $\bR_{R_{\qu}(T)}=R_\qu(\bR_T)$.
\end{enumerate}
The verifications of these results are elementary.
\end{remark}
The next proposition gathers some useful identities to prove the S-spectral properties of commutators which are provided in Theorem \ref{CT1}.
\begin{proposition}\label{p1} For arbitrary operators $S,T\in B(\vr)$ the following assertions hold true:
\begin{enumerate}
	\item[(a)] $\apo(\bL_S)=\apo(S)$,
	\item[(b)] $\apo(\bR_T)=\sus(T)$,
	\item[(c)] $\sus(\bL_S)=\sus(S)$,
	\item[(d)] $\sus(\bR_T)=\apo(T)$.
\end{enumerate}
\end{proposition}
\begin{proof}
~~(a) To prove (a), let $\qu\in\apo(\bL_S)$, then there exists a sequence $\{A_n\}\subseteq B(\vr)$ with $\|A_n\|=1$ such that $\|R_\qu(\bL_S)A_n\|\rightarrow 0$. That is, $\|R_\qu(S)A_n\|\rightarrow 0$. Set $\displaystyle \psi_n=\frac{\phi}{\|A_n\phi\|_{\vr}}$, for all $n$ and for some $0\not=\phi\in\vr$. Then $\|A_n\psi_n\|_{\vr}=1$, for all $n$ and $\|R_\qu(S)A_n\psi_n\|_{\vr}\rightarrow 0$. Thus $\qu\in\apo(S)$ and $\apo(\bL_S)\subseteq\apo(S)$. To see the other inclusion, let $\qu\in\apo(S)$, then there exists a sequence $\{\psi_n\}\subseteq\vr$ with $\|\psi_n\|_{\vr}=1$ and $\|R_\qu(S)\psi_n\|_{\vr}\rightarrow 0$. Pick a linear functional $\Phi$ in $(\vr)^*$ which is the dual of $\vr$ with $\|\Phi\|^*=1$, where $\|\cdot\|^*$ is a norm on the dual of $\vr$. Define, for each $n$, an operator $A_n\in B(\vr)$ by 
$$A_n\phi=\psi_n\Phi(\phi),\quad \text{for all}~~\phi\in\vr.$$ 
Then $\|A_n\|=1$ for all $n$, and
$$\|R_\qu(\bL_S)A_n\|=\|R_\qu(S)A_n\|=\|R_\qu(S)\psi_n\|_{\vr}\rightarrow 0.$$
Thus $\qu\in\apo(\bL_S)$ and therefore $\apo(\bL_S)=\apo(S).$\\	
~~(b) To establish the equality $\apo(\bR_T)=\sus(T)$, let $\qu\not\in\sus(T)$, that is $R_\qu(T)$ is surjective. Now for each $A\in B(\vr)$,
$$\|R_\qu(\bR_T)A\|=\|AR_\qu(T)\|\geq\|AR_\qu(T)\phi\|_{\vr},\quad\text{for all}~~\phi\in\vr.$$
That is, as $R_\qu(T)$ is surjective, $\|R_\qu(\bR_T)A\|\geq\|A\psi\|_{\vr}$, for all $\psi=R_\qu(T)\phi\in\vr$. Hence
$$\|R_\qu(\bR_T)A\|\geq\sup_{\|\psi\|=1}\|A\psi\|_{\vr}=\|A\|,\quad\text{for all}~~A\in B(\vr).$$
Therefore $R_\qu(\bR_T)$ is bounded below, and hence by proposition \ref{P4}, $\qu\not\in\apo(\bR_T)$. Conversely suppose that $R_\qu(\bR_T)$ is bounded below. Then there exists $c>0$ such that $c\|A\|\leq|\|R_\qu(\bR_T)\||=\|AR_\qu(T)\|$, for all $A\in B(\vr)$; where $|\|\cdot\||$ is the norm on $B(B(\vr))$. Choose a unit vector $\psi\in\vr$. For arbitrary linear functional $\Phi\in(\vr)^*$, let $A_\Phi\in B(\vr)$ given by
$$A_\Phi(\phi)=\psi\Phi(\phi),\quad\text{for all}~~\phi\in\vr.$$
Then
$$c\|\Phi\|^*=c\|A_\Phi\|\leq\|A_\Phi R_\qu(T)\|=\|\Phi\circ R_\qu(T)\|,\quad\text{for all}~~\Phi\in(\vr)^*.$$
Hence $R_\qu(T)^\dagger$ is bounded below. That is, by proposition \ref{NT1}, $R_\qu(T)^\dagger$ is injective. Therefore by propositions \ref{IP30}, $\ra(R_\qu(T))^\perp=\kr(R_\qu(T)^\dagger)=\{0\}$, and so $R_\qu(T)$ is surjective. Thus we have $\apo(\bR_T)=\sus(T).$\\
~~(c) Now 
\begin{eqnarray*}
\qu\notin\sus(\bL_S)&\iff& R_\qu(\bL_S)\text{~~is sujective}
\iff R_\qu(S)\text{~~is sujective}
\iff\qu\notin\sus(S).
\end{eqnarray*}
Therefore $\sus(\bL_S)=\sus(S)$.\\
~~(d) In order to verify the inclusion $\sus(\bR_T)\subseteq\apo(T)$, let $\qu\not\in\apo(T)$, then by proposition \ref{P4}, $R_\qu(T)$ is bounded below on $\vr$. Therefore, by proposition \ref{NT1}, $R_\qu(T)$ is injective. Thus by proposition \ref{PI}, $R_\qu(T)$ is left invertible on $\vr$. Therefore, there exist $P\in B(\vr)$ such that $PR_\qu(T)=\Iop$. This implies that $$R_\qu(\bR_T)\bR_PA=\bR_P(A)R_\qu(T)=APR_\qu(T)=A,\quad\text{for all}~~A\in B(\vr).$$
That is, $R_\qu(\bR_T)$ is right invertible on $B(\vr)$, thus by proposition \ref{PS}, $R_\qu(\bR_T)$ is surjective. Hence $\qu\not\in\sus(\bR_T)$, and we get $\sus(\bR_T)\subseteq\apo(T)$. To verify the other inclusion $\apo(T)\subseteq\sus(\bR_T)$, suppose that $\qu\not\in\sus(\bR_T)$, then $R_\qu(\bR_T)$ is surjective. This implies that for each $A\in B(\vr)$, there exists $B\in B(\vr)$ such that $R_\qu(\bR_T)B=A$. That is, $BR_\qu(T)=A$. Assuming $A,B\neq0$ without loss of generality, we get
$$\|R_\qu(T)\|\geq\frac{\|A\|}{\|B\|}.$$ This gives that $R_\qu(T)$ is bounded below, and $\qu\not\in\apo(T)$. Therefore the equality $\sus(\bR_T)=\apo(T)$ holds.
\end{proof}	
The following theorem is the main result about the S-spectral properties of commutators which we provide in this note. 
\begin{theorem}\label{CT1}
	For arbitrary operators $S,T\in B(\vr)$ the following assertions hold true for their commutator $C(S,T)$.
	\begin{enumerate}
		\item[(a)] $\sigma_S(C(S,T))=\sigma_S(S)-\sigma_S(T)$,
		\item[(b)] $\apo(C(S,T))=\apo(S)-\sus(T)$,
		\item[(c)] $\sus(C(S,T))=\sus(S)-\apo(T)$.
	\end{enumerate}
\end{theorem}
\begin{proof}
To prove (b),  In order to establish $\apo(S)-\sus(T)\subseteq\apo(C(S,T))$, let $\qu\in\apo(S)$ and $\pu\in\sus(T)$. It follows, from (a) and (b) in the proposition \ref{p1}, that $\qu\in\apo(\bL_S)$ and $\pu\in\apo(\bR_T)$. By proposition \ref{TC1} we have
$$\qu\in\apo(\bL_S)=\apo(\bL_S^\circ)=\sigma_{pS}(\bL_S^\circ).$$
Therefore, there exists $A\in B(\vr)$ such that $A\not=0$ and
$(\bL_S^\circ-\qu)A=0$. That is, $(S^\circ-\qu)A=0$. Again by proposition \ref{TC1},
$$\pu\in\apo(\bR_T)=\apo(\bR_T^\circ)=\sigma_{pS}(\bR_T^\circ).$$
Therefore there exists $B\in B(\vr)$ such that $B \not=0$ and $(\bR_T^\circ-\pu)B=0$. That is $B(T^\circ-\pu)=0$. Consider
\begin{eqnarray*}
	(C(S,T)^\circ-\qu+\pu)A B&=&(\bL_S^\circ-\bR_T^\circ-\qu+\pu)A B\\
	&=&(\bL_S^\circ-\qu)A B-(\bR_T^\circ-\pu)AB\\
	&=&(S^\circ-\qu)A B-AB(T^\circ-\pu)=0.
\end{eqnarray*}
Thus
$$\qu-\pu\in\sigma_{pS}(C(S,T)^\circ)=\apo(C(S,T)^\circ)=\apo(C(S,T)).$$
Therefore $\apo(S)-\sus(T)\subseteq\apo(C(S,T)).$ Now since $\apo(\bL_S)=\apo(S)$ and $\apo(\bR_T)=\sus(T)$, we get from part (a) of proposition \ref{COP1}, 
$$\apo(C(S,T))=\apo(\bL_S-\bR_T)\subseteq\apo(\bL_S)-\apo(\bR_T)\subseteq\apo(S)-\sus(T).$$
This concludes the proof for (b).\\	
~~(c) Applying $\sus(\bL_S)=\sus(S)$, and $\sus(\bR_T)=\apo(T)$ in part (b) of proposition \ref{COP1}, the inclusion $\subseteq$ in assertion (c) is established. Next to show that $\sus(S)-\apo(T)\subseteq\sus(C(S,T))$, let $\qu\in\sus(S)$ and $\pu\in\apo(T)$. It follows from (a) and (b) in proposition \ref{p1}, and proposition \ref{TC1} that
$$\qu\in\apo(\bR_S)=\apo(\bR_S^\circ)=\sigma_{pS}(\bR_S^\circ)\quad \text{and}\quad \pu\in\apo(\bL_T)=\apo(\bL_T^\circ)=\sigma_{pS}(\bL_T^\circ).$$
Thus by the definition of point spectrum 	$\oqu\in\sigma_{pS}(\bR_S^\circ)\quad \text{and}\quad \opu\in\sigma_{pS}(\bL_T^\circ).$
Therefore, there exists $A, B\in B(\vr)$ with $A\not=0$ and $B\not=0$ such that $A(S^\circ-\oqu)=0$ and $(T^\circ-\pu)B=0$. Hence, by proposition \ref{da}, we have $((S^\circ)^\dagger-\qu)A^\dagger=0$ and $B^\dagger((T^\circ)^\dagger-\pu)=0$. Hence, by proposition \ref{da},
\begin{eqnarray*}
	((C(S,T)^\circ)^\dagger-\qu+\pu)A^\dagger B^\dagger
	&=&((\bL_S^\circ)^\dagger-\qu)A^\dagger B^\dagger-((\bR_T^\circ)^\dagger-\pu)A^\dagger B^\dagger\\
	&=&((S^\circ)^\dagger-\qu)A^\dagger B^\dagger -A^\dagger B^\dagger((T^\circ)^\dagger-\pu)=0.
\end{eqnarray*}
Therefore $\qu-\pu\in\sigma_{pS}((C(S,T)^\circ)^\dagger)$. By proposition \ref{C1}, $\qu-\pu\in\sigma_{pS}((C(S,T)^\circ)^\dagger)=\sigma_{pS}((C(S,T)^\dagger)^\circ)$. Now by propositions \ref{TC1} and \ref{su1}, we have 
$$\qu-\pu\in \sigma_{pS}((C(S,T)^\dagger)^\circ)=\apo(C(S,T)^\dagger)=\sus(C(S,T)).$$
Hence  $\sus(S)-\apo(T)\subseteq\sus(C(S,T))$, which completes the proof of (c).\\
To establish (a), let $S, T\in B(\vr)$. Since $\bL_S\bR_T(A)=\bR_T\bL_S(A)=SAT$ for all $A\in B(\vr)$, $\bL_S$,$\bR_T\in B(B(\vr))$ commute. Let $\qu\in\sigma_S(\bL_S)$, then if $\kr(\bL_S)\not=\{0\}$, there exists  $A\in\B(\vr)$ such that $A\not=0$ and $R_\qu(\bL_S)(A)=0$. That is
$$S^2A-2\re(\qu)SA+|\qu|^2A=(S^2-2\re(\qu)S+|\qu|^2)A=0.$$
Hence $(S^2-2\re(\qu)S+|\qu|^2)A\phi=0$, for some $\phi\in\vr$ as $A\not=0$, and therefore $\kr(R_\qu(S))\not=\{0\}.$ If $\ra(R_\qu(\bL_S))\not= B(\vr)$, then there exists $B\in B(\vr)$ such that $R_\qu(\bL_S)(A)\not=B$ for all $A\in B(\vr)$. That is, $S^2A-2\re(\qu)SA+|\qu|^2A\not=B$ for all $A\in B(\vr)$. In other words, $R_\qu(S)A\phi\not=B\phi$ for all $A\in B(\vr)$ and $\phi\in\vr)$. Hence $\ra(R_\qu(S))\not=\vr$. As a conclusion $\qu\in\Ss(S)$ and hence $\Ss(\bL_S)\subseteq\Ss(S)$.\\
Now let $\qu\in\Ss(\bR_T)$. If $\kr(R_\qu(\bR_T))\not=\{0\}$, then there exists $A\in B(\vr)$ such that $A\not=0$ and $R_\qu(\bR_T)(A)=0$, that is $AR_\qu(T)=0$. Thus $R_\qu(T)\phi=0$ for some $0\not=\phi\in\vr$, and therefore $\kr(R_\qu(T))\not=\{0\}$. If $\ra(R_\qu(\bR_T))^\bot\not= B(\vr)$, then there exists $B\in\B(\vr)$ such that $R_\qu(\bR_T)(A)\not=B$, for all $A\in B(\vr)$. That is $AR_\qu(T)\not=B$, for all $A\in B(\vr)$, and hence $\Iop R_\qu(T)\not=B$. Therefore $\ra(R_\qu(T))\not=\vr$. Hence we can conclude that $\qu\in\Ss(T)$ and $\Ss(\bR_T)\subseteq\Ss(T)$. Because $C(S,T)=\bL_S-\bR_T$, by part (c) of proposition \ref{COP1} we have
$$\Ss(C(S,T))\subseteq\Ss(\bL_S)-\Ss(\bR_T)\subseteq\Ss(S)-\Ss(T).$$
This establishes the inclusion $\subseteq$ in assertion (a). Now for each $A,B\in B(\vr)$, we have
\begin{eqnarray*}
	\bC(\bL_S,\bR_T)AB&=&\bL_SAB-A\bR_TB
	=SAB-ABT
	=C(S,T)AB.
\end{eqnarray*}
This implies that
\begin{equation}\label{Comeq}
	\bC(\bL_S,\bR_T)=C(S,T), ~\forall\,S,T\in B(\vr).
\end{equation} On the other hand, using equation \ref{Comeq}, from part (c),
\begin{equation}\label{e1}
\sus(C(S,T))=\sus(\bC(\bL_S,\bR_T))\supseteq\sus(\bL_S)-\apo(\bR_T)=\sus(S)-\sus(T)
\end{equation} 
as 	$\apo(\bR_T)=\sus(T)$ and $\sus(\bL_S)=\sus(S)$. Similarly from part (b), we get
\begin{equation}\label{e2}
\apo(C(S,T))=\apo(\bC(\bL_S,\bR_T))\supseteq\apo(\bL_S)-\sus(\bR_T)=\apo(S)-\apo(T)
\end{equation}
as 	$\apo(\bL_S)=\apo(S)$ and $\sus(\bR_T)=\apo(T)$. Now the inclusions (\ref{e1}) and (\ref{e2}) guarantee that the other inclusion in assertion (a) holds,  $$\Ss(C(S,T))\supseteq\Ss(\bL_S)-\Ss(\bR_T)\subseteq\Ss(S)-\Ss(T).$$ Therefore the assertion (a) follows. Hence the theorem holds.


\end{proof}
\section{Acknowledments}
K. Thirulogasanthar would like to thank the FRQNT, Fonds de la Recherche  Nature et  Technologies (Quebec, Canada) for partial financial support under the grant number 2017-CO-201915. Part of this work was done while he was visiting the University of Jaffna to which he expresses his thanks for the hospitality.

\end{document}